\documentclass[bjps]{imsart}

\RequirePackage[OT1]{fontenc}
\usepackage{amsthm,amsmath,natbib}
\usepackage{ upgreek }
\usepackage{mathrsfs}

\usepackage{amsfonts}
\usepackage{amssymb}
\RequirePackage{hyperref}

\arxiv{arXiv:0000.0000}

\startlocaldefs
\numberwithin{equation}{section}
\theoremstyle{plain}
\newtheorem{thm}{Theorem}[section]
\theoremstyle{definition}
\newtheorem{definition}{Definition}[section]
\newtheorem{remark}{Remark}[section]
\newtheorem{lemma}[thm]{Lemma}
\newtheorem{corollary}[thm]{Corollary}
\newtheorem{proposition}[thm]{Proposition}\endlocaldefs

\def\si{{\sigma}}
\def\la{{\lambda}}

\def\ga{{\gamma}}
\def\Ga{{\Gamma}}
\def\om{{\omega}}
\def\Om{{\Omega}}
\def\eps{{\epsilon}}

\newcommand{\und}{\underline}
\newcommand{\dis}{\displaystyle}

 \newcommand{\nn}{\nonumber}

\begin{document}

\begin{frontmatter}
\title{Separation versus  diffusion in a two species system}
\runtitle{A two species system}

\begin{aug}
\author{\fnms{Anna} \snm{De Masi}\thanksref{a}
\ead[label=e1]{demasi@univaq.it}}
\author{\fnms{Pablo A} \snm{Ferrari}
\thanksref{b}
\ead[label=e2]{pferrari@dm.uba.ar}}

\runauthor{A. De Masi and P.A. Ferrari}

\affiliation[a]{Universit\`a L'Aquila}
\affiliation[b]{Universidad de Buenos Aires and IMAS CONICET}

\address{Dipartimento DISIM
\\ Universit\`a L'Aquila\\
Via Vetoio,1 \\67100 L'Aquila (Italy)\\
\printead{e1}}

\address{IMAS-CONICET and Departamento de Matem\'atica \\ Universidad de Buenos Aires
\\Pabell\'on 1
\\ 1428 Ciudad Aut\'onoma de Buenos Aires
(Argentina)\\
\printead{e2}}

\end{aug}

\begin{abstract}
We consider a finite number of particles that move in $\mathbb Z$ as independent random walks. The particles are of two species that we call $a$ and $b$. The rightmost $a$ particle becomes a $b$ particle at constant rate, while the leftmost $b$ particle becomes $a$ particle at the same rate, independently. We prove that in the hydrodynamic limit the evolution is described by a non linear system of two PDE's  with free boundaries.
\end{abstract}

\begin{keyword}[class=MSC]
\kwd[Primary ]{60K35}
\kwd{82C22}
\kwd[; secondary ]{65N75}
\end{keyword}

\begin{keyword}
\kwd{hydrodynamic limit}
\kwd{interacting particle systems}
\kwd{free boundaries PDE}
\end{keyword}


\end{frontmatter}

\section{Introduction}
We consider a two-species particle system in $\mathbb Z$, the species, also called colors, are indicated by  $a$ and $b$. We suppose that at time 0 the species are
partially separated with a rightmost $a$-particle at a site denoted by $X_0$ and  a leftmost  $b$-particle
at a site  $Y_0< X_0$.
The evolution is such that if we are ``color blind"
we just see independent symmetric random walks which jump at rate one on the nearest neighbor sites.
As particles keep their color during their random walk motion this means that the $a$ and $b$ species diffuse
in $\mathbb Z$. In our model however particles
may also change color with the following mechanism.
Independently at rate $\la>0$ the rightmost $a$-particle becomes a
$b$-particle and the leftmost $b$-particle becomes an $a$-particle.
If the evolution consisted only of
this color exchanges, then eventually $a$ and $b$ would separate, but this is contrasted
in our model by the random walk motion of the particles which drives toward homogenization.

The motivation behind this paper is to understand
how much the species separate
as time evolves when both random walks and color exchange are acting,
in particular to determine
the evolution of
the difference $X_t-Y_t$, with $X_t$ and $Y_t$ the positions at time $t$ of the rightmost $a$-particle and leftmost $b$-particle respectively.
In this paper we begin this program by looking at the
hydrodynamic scale: we take $\la=\eps \kappa$, $\kappa>0$, and scale space and time diffusively
($x\to r=\eps x$, $x\in\mathbb Z$,  $t\to \tau=\eps^2 t$).
We assume that the initial distribution is such that the densities of the two species approach in the limit $\eps\to 0$ a macroscopic profile and that the total mass is macroscopically finite. These two assumptions imply that the total number of particles is of order $\eps^{-1}$, see Section \ref{sec1} below for a precise definition of the initial condition.

Under the above hypothesis we prove convergence
as $\eps\to 0$ to a non linear system of two PDE's  with free boundaries.

Despite its simplicity the rule at which species mutate creates a very non local interaction: to find the rightmost b-particle it is necessary to know the whole configuration of $b$-particles. 
Here the interaction is Òtopological rather than metricÓ, as the influence on a particle $i$ of a particle $j$ does not depend on their distance but rather
depends on whether $j$ is to the right or left of $i$. Stochastic evolutions with similar non local interactions have been considered to model problems from  different fields such as queuing theory, \cite{ABK}, statistical mechanics of open systems (currents and Fourier law), \cite{CDGP1}, \cite{DPTV} and pinned interface motions, \cite{lacoin}.

\section{Model and Results}
\label{sec1}

We thus consider a system of colored particles on $\mathbb Z$. Both the initial distribution
and evolution depend on a scaling parameter $\eps>0$. We are interested in the hydrodynamic limit when $\eps\to 0$ and space and time are rescaled diffusively.

{\bf The initial condition}.
The initial macroscopic profile is described
by a pair $(u,v)$ of non negative functions on
$\mathbb R$ which are interpreted as
the macroscopic particle densities of the $a$
and respectively $b$ species.
We suppose  that
$(u,v)\in \mathcal U$:
\begin{eqnarray}
\nn &&\hskip-2cm	\mathcal U=\Big\{ (u,v)\in  C_0(\mathbb R,\mathbb R^2_+) :
\text{support } u=(L,R), \; \text{support } v=(D,E);
\\&&\hskip1.5cm \; L<D<R<E,\;\; u,v >0\text{ in their support }
 \Big\}.
 \label{1.4}
 \end{eqnarray}
The total ``macroscopic mass'' of the two species is denoted by
\[
M_{\rm tot} = \int (u+v).
\]

The macroscopic profiles $(u,v)$ are approximated by particle
configurations using a scaling parameter $\eps>0$.
For each $\eps>0$ the initial  configuration has
$M:= [\eps^{-1}M_{\rm tot}]$ particles.
Their positions $\und x= (x_1,\dots,x_M)$ are random, they are independently identically distributed with parameters
		\begin{equation}
	\label{4.18}
P^\eps[ x_i=x] = Z_\eps^{-1}[u(\eps x)+v(\eps x)],\quad
Z_\eps=\sum_x [u(\eps x)+v(\eps x)].
	\end{equation}
Conditioned on $\und x$ we add independently a color
$\si_i \in \{a,b\}$ to
each particle $i$, by setting
	\begin{equation}
	\label{4.17}
P^\eps[ \si_i=a \,|\, \und x] = \frac {u(\eps x_i)}{u(\eps x_i)+v(\eps x_i)}.
	\end{equation}
It is convenient for technical purposes to
label the particles
but the physically relevant quantities are the occupation numbers
 	\begin{equation}
	\xi_{\und x,\und \si}(y)=\sum_{i=1}^{M} \mathbf 1_{x_i =y,\si_i =a},\qquad 	\eta_{\und x,\und \si}(y)=\sum_{i=1}^{M} \mathbf 1_{x_i =y,\si_i =b},\quad y\in\mathbb Z,
		\end{equation}
where  $\und \si= (\si_1,..,\si_M)$. We then say that  $(\und x,\und \si)$ and $(\und x',\und \si')$
are equivalent if
	\begin{equation}
\label{2.3n}
 \xi_{\und x,\und \si}=\xi_{\und x',\und \si'},\quad \eta_{\und x,\und \si}=\eta_{\und x',\und \si'},
		\end{equation}
which means that
one can be obtained from the other by exchanging colors of particles at the same site.

It easily follows from the above definitions that under $P^\eps$,
$ (\eps\xi_{\und x,\und \si},\eps \eta_{\und x,\und \si})$ converges weakly in probability
to $(u,v)$
as $\eps\to 0$.  Our main results will be to extend the result to positive times
and identify the limit.

{\bf The positions time evolution}.  If we disregard the color of the particles we just see a system of independent random walks
denoted by $\und x(t)= (x_1(t),\dots,x_M(t))$, $t\ge 0$. The $x_i(t)$ are symmetric
independent random walks on $\mathbb Z$ which jump at rate 1
on nearest neighbor sites. We denote by $\mathscr P^\eps$ the law of this process.

We shall next define how the colors change
in time. To this end we first define the label of the rightmost $a$
and leftmost $b$ particles denoted respectively by $i_a(\und x,\und\si)$ and $i_b(\und x,\und\si)$.

 	\begin{definition}
	\label{def2.0}
We denote  the total number of $a$ and $b$ particles respectively by
	\begin{equation}
	\label{2.4ne}
h_a(\und \si)=\sum_i \mathbf 1_{\si_i=a},\qquad h_b(\und\si)=M-h_a(\und \si).
	\end{equation}
If $h_a(\und\si)>0$ we define
$i_a(\und x,\und\si)=i$ if $\si_i=a$ and for any  $j\ne i$ with  $\si_j=a$,   either $x_j<x_i$ or, if $x_j=x_i$, then
$j<i$. Analogously if $h_b(\und\si)>0$ $i_b(\und x,\und\si)=i$ if $\si_i=b$ and if $\si_j=b$,  either $x_j>x_i$, or if $x_j=x_i$, then
$j<i$.
We also define the operators
$H^{\rm right}(\und x,\und \si) =: (\und x,\und \si')$,
$H^{\rm left}(\und x,\und \si) =: (\und x,\und \si'')$
where $\und \si'=\si$ if $h_a(\und \si)=0$, $\und \si''=\si$ if $h_b(\und \si)=0$. Otherwise $\und \si'$ and  $\und \si''$ are obtained from $\und \si$ by changing $\si_{i_a(\und x,\und\si)}$ into $b$ and respectively $\si_{i_b(\und x,\und\si)}$ into $a$.

\end {definition}

The evolution of colors is determined
by the clock rings of the following Poisson processes.

	\begin{definition}
	\label{def3.0}
Given  $\eps>0$ and $j>0$ we define the probability space $(\Om, \mathbb P^\eps)$. $\Om$ is the set of $\om=(\und s,\und \ell)$ where   $\und s=(s_1,s_2,\dots)$ $s_k\le s_{k+1}$ is an ordered  sequence of times, and    $\und \ell=(\ell_1,\ell_2,\dots)$, $\ell_k\in\{$right, left$\}$ is a sequence of marks. $\mathbb P^\eps$ is the  product probability {law} of a Poisson
{process} of intensity $2\eps \kappa$ for
the time sequences $\und s$ and of a
Bernoulli
{process} {with parameter $1/2$} for
the mark sequences $\und \ell.$
In the sequel we will consider strictly increasing  sequences of time $\und s$ since these have $\mathbb P^\eps$- probability one. We denote by $\mathcal P^\eps=\mathscr P^\eps\times \mathbb P^\eps$ the joint law of the random walk $\und x$ and of $\om$.
\end {definition}

{\bf The color time evolution}. 
Given $\eps>0$, $\und x(t)$, $t\ge 0$, and $\om=(\und s,\und \ell)$
we define the ``c\`adl\`ag'' trajectory $\und \si(t)$ by saying that colors are unchanged except
at the times $s_k$:
at these times the  configuration is updated by applying $H^{\rm right}$ or  $H^{\rm left}$ according to $\ell_k=$ right or  $\ell_k=$  left, respectively.
We denote by  $(\und x(t),\und\si(t))$ positions and colors of particles at time $t$.

The main results in this paper are  Theorems \ref{teo1} and \ref{teo2} below.

\begin{thm}
\label{teo1}
Under the above assumptions on the initial data there are non negative continuous
functions  $(\bar u(\cdot,t), \bar v(\cdot,t))$ equal to $(u,v)\in\mathcal U$ at $t=0$ and such that
for any $t>0$
\[
\Big(\eps\, \xi_{\und x(\eps^{-2}t),\und \si(\eps^{-2}t)},
\eps \,\eta_{\und x(\eps^{-2}t),\und \si(\eps^{-2}t)}\Big) \to (\bar u(\cdot,t), \bar v(\cdot,t)),
 \]
as $\eps\to 0$ weakly in probability.

	\end{thm}
Since for all $s\ge 0$, $ \xi_{\und x(s),\und \si(s)}(x)+
\eta_{\und x(s),\und \si(s)}(x) = \sum_i \mathbf 1_{x_i(s)=x}$
and the $x_i(\cdot)$ are independent random walks, we know, see for instance \cite{dp1}, that
\[
\eps\Big[ \xi_{\und x(\eps^{-2}t),\und \si(\eps^{-2}t)} +
\eta_{\und x(\eps^{-2}t),\und \si(\eps^{-2}t)}\Big] \to w(\cdot,t),
\]
as $\eps\to 0$ weakly in probability with $w$ the solution of the linear heat equation
$w_t= \frac 12 w_{rr}$ and initial condition $u+v$.  Thus it is enough for Theorem \ref{teo1}
to prove convergence of $ \xi_{\und x(\eps^{-2}t),\und \si(\eps^{-2}t)}$ alone.

 The proof is reported in Section \ref{sec.3}, it follows  the
same strategy
used in \cite{dfp} and then in \cite{CDGP1}. Namely we first introduce auxiliary processes
for which the hydrodynamic limit can be computed and then prove by stochastic inequalities
that the true process is sandwiched  between the auxiliary ones and
that the inequalities
become equalities in the limit.  The first part is easy (as the auxiliary processes
are essentially independent random walks) and we just sketch it in Section \ref{sec.3}.
The proof of the stochastic inequalities is instead quite involved and
given in full details
in the next section, being one of the most important parts of the paper.

Theorem  \ref{teo1} only states  the  existence
of the hydrodynamic limit for all macroscopic times $t\ge 0$.
It does not give its properties nor specifies the hydrodynamic equations.
On  the other hand one may guess that the latter are given by
the following system of two equations
	\begin{eqnarray}
	\nn
&&
\hskip-1cm
u_t=\frac 12 u_{rr}+\kappa\,\delta_{V_t},\; r < U_t; \; u(r,0)=u (r), \;u(U_t,t)=0, \; -\frac 12 u_r(U^-_t,t)=\kappa,
\\   \label{eq11}
\\&&
\hskip-1cm v_t=\frac 12 v_{rr}+\kappa\,\delta_{U_t},\; r > V_t; \; v(r,0)=v (r),\; v(V_t,t)=0, \;   -\frac 12 v_r(V^+_t,t)=-\kappa, \nonumber
	\end{eqnarray}
where $U_0=R$ and $V_0=D$, see \eqref{1.4}, and   $\delta_x$ is the Dirac delta.

\eqref{eq11} is a system of two free boundary equations as the domains
$(-\infty, U_t)$ where $u(r,t)$ is defined
and $(V_t,\infty)$ where $v(r,t)$ is defined are also unknowns to be determined.

By the Dirichlet condition $u(r,t)$ can be extended continuously past $U_t$ by setting $u(r,t)\equiv 0$ for all $r\ge U_t$ so that
 $U_t$ is the rightmost-end point of the interval where $u>0$, it thus corresponds to the macroscopic
position of the rightmost particle.  Analogous interpretation is given to $V_t$. In the particle
system $a$-particles are created at rate $\eps \kappa$ at the position of the leftmost $b$-particle,
correspondingly the equation for $u$ has a source term
$\kappa\,\delta_{V_t}$, with an analogous interpretation
for $\kappa\,\delta_{U_t}$.  Finally the boundary condition $-\frac 12 u_r(U^-_t,t)=\kappa$
just says that the outgoing mass flux of $u$ is equal to $\kappa$ which is the macro-analogue of the rate at which
$a$ particles disappear (changing into $b$ particles), analogous interpretation holds for
the term $-\frac 12 v_r(V^+_t,t)=-\kappa$.

The two equations are coupled by the Dirac-delta terms which involve the free boundary terms $U_t$ and $V_t$ which make the problem highly non linear.

We did not find in the  literature the above system of free boundary  problems. 
We
notice however that \eqref{eq11} is similar 
to the free boundary PDE studied and
for which local and
sometimes global existence and uniqueness are proved, see for instance \cite{fasano}.  
It is then conceivable  that the same techniques might be applied to our equation, but we did not pursue this issue here, so we assume existence 
of a solution 
and we prove that this solution coincides with the limit of our particles evolution.

We thus suppose that for some positive
time interval $[0,T]$ there is a regular solution of  \eqref{eq11}. By regular we mean that the functions  $U_t,V_t$ of \eqref{eq11} 
are
$C^1[0,T]$; that  $u(r,t)$, $v(r,t)$ have the differentiability properties required by
\eqref{eq11}, and finally that $(u(\cdot,t),v(\cdot,t))\in\mathcal U$ for all   
 $t\in[0,T]$.

\begin{thm}
\label{teo2}
Assume there is $T>0$ so that a regular solution of \eqref{eq11} exists in the above sense in $[0,T]$.
Then this solution  coincides with the hydrodynamic limit $(\bar u(\cdot,t), \bar v(\cdot,t))$ of Theorem \ref{teo1} restricted to $t\in[0,T]$.

\end{thm}

We prove Theorem \ref{teo2} in Section \ref{sec5}.
The proof has some similarities with the proofs in  \cite{CDGP1} but it requires new ideas
and it is the other   most important point of the paper together with the proof
of the microscopic inequalities.

\section{Microscopic inequalities}
	\label{sect4}

As already mentioned stochastic inequalities play a fundamental role in our proof.
Let $\xi'$ and $\xi$ be non negative, integer valued functions on $\mathbb Z$ with compact support.
		\begin{definition}
\label{def4.1}
 We say that $\xi'\preccurlyeq \xi$  if  for all $x\in \mathbb Z$
	\begin{equation}
	\label{4.1}
F(x;\xi')\le F(x;\xi),\qquad F(x;\xi)=\sum_{y\ge x} \xi(y).
	\end{equation}
We also say that  $(\und x',\und \si')\preccurlyeq (\und x,\und \si)$
if $\xi_{\und x',\und \si'} \preccurlyeq \xi_{\und x,\und \si}$ (observe
that the inequality remains valid if we replace a configuration by an
equivalent one, see \eqref{2.3n}).
\end{definition}

Recalling Definitions \ref{def2.0} and \ref{def3.0} we first introduce the following sets.

\begin{definition} 
	\label{de4.0}
We call  $\mathcal X_t$, $t>0$ the set of all $(\und \si, \om)$ such that
\begin{eqnarray}
\label{3.4}&&
N_a(t):= h_a(\und \si) + \sum_k   \mathbf 1_{s_k\le t}\Big( \mathbf 1_{\ell_k=\text{left}} -
\mathbf 1_{\ell_k=\text{right}}\Big)>0,\quad
\\&&
N_b(t):= M-h_a(\und \si)  + \sum_k  \mathbf 1_{s_k\le t} \Big( \mathbf 1_{\ell_k=\text{right}} -
\mathbf 1_{\ell_k=\text{left}}\Big) >0.
\nn 
\end{eqnarray}
\end{definition} 
In $\mathcal X_t$ there are always both $a$ and $b$ particles
in the time interval $[0,t]$.
In the next section, see Lemma \ref{lem4.1}, we prove that $\mathcal P^\eps(\mathcal X_t)\to 1 $ as $\eps\to 0$.

{\bf The setup.} 
 Throughout this section we fix $\eps>0$, $\delta>0$,
a time interval $[0,\eps^{-2}\delta]$,   a random walk
trajectory $\und x(t)= (x_1(t),\dots,x_M(t))$, $t\in [0,\eps^{-2}\delta]$ and an element 
$(\und \si, \om)\in \mathcal X_{\eps^{-2}\delta}$.

{\bf
The auxiliary evolutions.} They are denoted by $(\und x(t), \und \si^{(\delta,\pm)}(t))$ and
 are defined 
by anticipating
or postponing the color changes at the initial, respectively final, time.  Namely, given $(\und \si, \om)\in \mathcal X_{\eps^{-2}\delta}$ we let  
$\und \si^{(\delta,+)}(t)$ be the function left continuous with right limits obtained by setting $\und \si^{(\delta,+)}(t)=\und \si^{(\delta,+)}(0^+)$ for $t\in(0,\eps^{-2}\delta]$
and
	\begin{equation}
(\und x(0^+),\und \si^{(\delta,+)}(0^+)) =  \prod_{i=1}^m H^{\ell_i}(\und x(0),\und \si).
  \label{4.15}
  \end{equation}
 Similarly $\und \si^{(\delta,-)}(t)$ is the function right continuous with left limits obtained by setting  $\und \si^{(\delta,-)}(t)=\und\si$ for all $t\in[0,\eps^{-2}\delta)$ while
 	\begin{equation}
	(\und x(\eps^{-2}\delta),\und \si^{(\delta,-)}(\eps^{-2}\delta)) =
  \prod_{i=1}^m H^{\ell_i}(\und x(\eps^{-2}\delta),\und \si).
  \label{4.16}
  \end{equation}

 \begin{thm}
 \label{teo3.1}
In the above setup, given $(\und \si, \om)\in \mathcal X_{\eps^{-2}\delta}$ and $\und \si'$ such that $h_a(\si')=h_a(\si)$ and  $(\und x(0),\und \si')\preccurlyeq (\und x(0),\und \si))$, we have 
	\begin{equation}
	\label{3.2}
 (\und x(\eps^{-2}\delta), {\si'}^{(\delta,-)}(\eps^{-2}\delta))
 \preccurlyeq (\und x(\eps^{-2}\delta), \si(\eps^{-2}\delta)),
	\end{equation}
	\begin{equation}
	\label{3.3}
 (\und x(\eps^{-2}\delta), {\si'}(\eps^{-2}\delta))
 \preccurlyeq (\und x(\eps^{-2}\delta), \si^{(\delta,+)}(\eps^{-2}\delta)).
	\end{equation}
In \eqref{3.2} $ {\si'}^{(\delta,-)}(\eps^{-2}\delta)$ is the 
auxiliary evolution associated to $(\und \si',\om)$, and in
\eqref{3.3}  $ {\si'}(\eps^{-2}\delta)$
 is the true evolution associated to $(\und \si',\om)$. The evolutions on the right hand side of \eqref{3.2} and \eqref{3.3} are respectively the true and the auxiliary $(\delta,+)$-evolutions associated to $(\und \si,\om)$.
 \end{thm}

We shall prove Theorem \ref{teo3.1} in the remaining part of this section by
constructing joint processes (that we call couplings by an abuse of notation) which exploit the fact that
the above inequalities remain valid
if we exchange colors of particles at the same site.

The coupling
is determined by specifying the colors of each $x_i(t)$ in the two processes, the one associated to $(\und \si',\om)$ and the one associated  to $(\und \si,\om)$: they have same positions and same $\om$. Thus the configurations in the coupled process are systems
of particles with two colors:
$(\und x, \Sigma)$, $\Sigma=(\und\si,\und\si')$. We call $(x_i,\si_i,\si'_i)$,
$i=1,., M$,
the specification of particle $i$.
With the aim
of establishing stochastic inequalities we split the particles of $(\und x,\Sigma)$ into
``married pairs'', ``singletons'', and ``discrepancies'' using the following notions:

\begin{itemize}
\item $i$ is a $a$-singleton or a $b$-singleton if it
has specification $(x_i, a, a)$, respectively  $(x_i,b,b)$;

\item $i$ is married with $j$ if $i$ has specification $(x_i,a,b)$ and
$j$ has specification  $(x_j,b,a)$  with $x_i>x_j$; $(i,j)$  are then said to be a ``married pair''.

\item $i$ is a $(b,a)$-discrepancy or a $(a,b)$-discrepancy if
it has specification $(x_i,b,a)$ or  $(x_i,a,b)$ respectively and it is not in a married
pair.
\end{itemize}

We  shall say  that a quadruple $(P,S,I,J)$ is a
``splitting''
of  $(\und x, \Sigma)$  if $P$ is a set of  married pairs,
$S$ a set of singletons, $I$ a set of $(b,a)$ discrepancies,
$J$ a set of $(a,b)$ discrepancies and each particle is either
in one (and only one) of the pairs in $P$ or if it is not in any of the pairs
then it is
in one (and only one) of the other three sets.  Of course
there are in general many ways to split $(\und x, \Sigma)$
into a quadruple $(P,S,I,J)$,   we want splittings with as less discrepancies as possible,
as it follows from the following lemma which will be
extensively used in the sequel (its proof is an immediate consequence of
the definitions and  omitted).

\begin{lemma}
\label{lem4.1}
Let $(P,S,I,J)$ be a splitting of
$(\und x, \Sigma)$, $\Sigma=(\und \si,\und \si')$ with  $I=J=\emptyset$.
Then
$(\und x, \und \si')  \preccurlyeq
(\und x, \und \si) $.  Viceversa if $(\und x, \und \si')  \preccurlyeq
(\und x, \und \si) $ there exists a splitting of  $(\und x, \Sigma)$, $\Sigma=(\und \si,\und \si')$, such that
$I=J=\emptyset$.

\end{lemma}

The coupling will be defined by specifying the  evolution
 $(\und x(t),\Sigma(t))$ and its splitting
$(P(t),S(t),I(t),J(t))$.

\noindent  {\bf The map $R$.}
Let $(s_k,s_{k+1})$ be an interval between events of the Poisson process and let
$(P,S,I,J)$ the quadruple at time $s_k$. Let $t^*$ be  the
first time after $s_k$ when $x_i(t^*)= x_j(t^*)$ for some $(i,j)  \in P$. We then set
 $(P(t),S(t),I(t),J(t))=(P,S,I,J)$ for $t < \min\{t^*,s_{k+1}\}$ and if $t^*<s_{k+1}$ we set
 $P(t^*)=P\setminus (i,j)$
and put $i,j \in S(t^*)$ with $i$  a $a$-singleton and $j$ a $b$-singleton (we have used here the fact
that we may exchange colors of particles at a same site). By iteration
the evolution  is extended till time $ s_{k+1}$ with a new configuration $\und x'$
and with a new splitting $(P',S',I',J')$.

The set of possible $\und x',(P',S',I',J')$ obtained in this way   is characterized
by the following requests: $I'=I$, $J'=J$
$P' \subseteq P$ with  $S'\setminus S$  made by all labels $i$ and $j$ of the pairs
which have disappeared.  $\und x'$ has the only constraint that $x'_i > x'_j$ if
$(i,j)\in P'$.  We denote by  $\mathcal R$ the collection of all maps $R$ such that
$R(\und x, P,S,I,J)$ has the  above properties.
The important points for the sequel are: (i) The discrepancies are unchanged under any $R\in \mathcal R$
 and  (ii) the identity map is in $\mathcal R$.

\noindent  {\bf The $C$-maps.} They describe the changes of colors which involve, according to cases, the
particles $i_a(\und x,\si)$, $i_b(\und x,\si)$, $i_a(\und x,\si')$
and $i_b(\und x,\si')$. Due to such changes
 the splitting quadruple $(P;S;I;J)$ associated to $(\und x,\Sigma)$ will be modified into
 a new quadruple $(P';S';I';J')$, in the way described below:

\begin{enumerate}
\item[$C^{\rm right}_1$:] shorthand $i=i_a(\und x,\si)$
	\begin{enumerate}
	\item if there is $j$  such that $(i,j)\in P$ then $P' = P \setminus (i,j)$,
$S'= S\cup i$, $I'= I \cup j$, $J'=J$.
	\item if $i\in {S}$ then $S'=S\setminus i$, $I' = I \cup i$, $J'=J$, $P'=P$.
	\item if $i\in J$  then $S'=S\cup i$, $J'=J\setminus i$, $I' = I $ and $P'=P$.
	\end{enumerate}

\item[$C_1^{\rm left}$:] shorthand $i=i_b(\und x,\si)$
	\begin{enumerate}
	\item if there is $j$  such that $(j,i)\in P$, then $P' = P \setminus (j,i)$,
$S'= S\cup i$, $J'= J \cup j$, $I'=I$.
	\item if $i\in {S}$ then $S'=S\setminus i$, $J' = J \cup i$, $I'=I$, $P'=P$.
	\item if $i\in I$  then $S'=S\cup i$, $I'=I\setminus i$, $J' = J $ and $P'=P$.
	\end{enumerate}

\item[$C^{\rm right}_2$:]  shorthand $i=i_a(\und x,\si')$ and $k$ the largest label in $I$
if  $I\ne \emptyset$
\begin{enumerate}
	\item  if there is $j$  such that $(j,i)\in P$  and $I\ne \emptyset$,
then $P' = P \setminus (j,i) \cup (j,k)$,
$S'= S \cup i$, $I'= I \setminus k$, $J'=J$;
	if instead $I= \emptyset$, then $P' = P \setminus (j,i)$,
$S'= S \cup i$, $I'=  I= \emptyset$, $J'=J\cup j$.
	\item if $i\in {S}$ and $I\ne \emptyset$
then $S'=S\setminus i$, $I' = I \setminus k$, $J'=J$, $P'=P\cup (i,k)$; if instead
 $I= \emptyset$,
then  $S'=S\setminus i$, $P'=P$, $I'=I$ and $J'= J\cup i$.
	\item if $i\in I$ then $I' = I \setminus i$, $P'=P$, $S'=S\cup i$ and $J'=J$.
	\end{enumerate}
\end{enumerate}

\begin{enumerate}

\item[$C_2^{\rm left}$:] shorthand $i=i_b(\und x,\si')$ and $k$ the largest label in $J$
if  $J\ne \emptyset$
\begin{enumerate}
	\item if there is $j$  such that $(i,j)\in P$  and $J\ne \emptyset$,
then $P' = P \setminus (i,j) \cup (k,j)$,
$S'= S \cup i$, $J'= J \setminus k$, $I'=I$;
	if instead $J= \emptyset$, then $P' = P \setminus (i,j)$,
$S'= S \cup i$, $J'= J= \emptyset$, $I'=I\cup j$.

\item if $i\in {S}$ and $J\ne \emptyset$
then $S'=S\setminus i$, $J' = J \setminus k$, $I'=I$, $P'=P\cup (k,i)$; if instead
 $J= \emptyset$,
then  $S'=S\setminus i$, $P'=P$, $J'=J$ and $I'= I\cup i$.
	\item if $i\in J$ then $J' = J \setminus i$, $P'=P$, $S'=S\cup i$ and $I'=I$.
	\end{enumerate}
\end{enumerate}

\begin{remark}
 The subscript $1$, $2$, reminds that the $C$
operator acts on the first component $\si$,
respectively the second one, $\si'$.  The above properties of the $C_2$ operators follow from the definitions of $i_a$ and $i_b$ allowing for the formation of married pairs which are instead not used  for the $C_1$ operators.  Recall that our goal is to prove that at the end $I$ and $J$ are empty, in this respect the $C_1$ operators are dangerous, as they may increase by 1 the cardinality of $I$ (with  $C^{\rm right}_1$) or $J$  (with  $C^{\rm left}_1$) while
the $C_2$ are recovery operators as they decrease by 1 the cardinality of $I$ (with  $C^{\rm right}_2$)  or $J$  (with  $C^{\rm left}_2$) when $I$ and $J$ are non empty.  This is behind the proof of the next theorem which, as we shall see after its proof, yields as a corollary the proof of Theorem \ref{teo3.1}.

\end{remark}

    \begin{thm}
    \label{teo4.2}
    Let $(P,S,I,J)$, $I=J=\emptyset$, be a quadruple associated
    to $(\und x, \Sigma)$.  Then for any non negative integer
    $m$, any sequence $(R_1,\dots,R_{m})$,
     $(R'_1,\dots,R'_{m})$ of elements of $\mathcal R$,        \begin{equation}
  \label{4.12}
(\und x^*, P^*,S^*,I^*,J^*):=( C_2 R')_m( C_1 R)_m(\und x, P,S,I,J)
   \end{equation}
 has $I^*=J^*=\emptyset$    where we have used the notation for $q\le m$: $( C_1 R)_q= C_1^{\ell_q}R_q\cdots  C_1^{\ell_1}R_1$
    and  $( C_2 R')_q= C_2^{\ell_q}R'_q\cdots  C_2^{\ell_1}R'_1$.

  \end{thm}

\begin{proof} Observe that the elements of $\mathcal R$ change only the sets $P$ and $S$, thus to prove the Theorem we only need to consider the $C$-maps.
For $q\le m$ we call  $I_q$ and $J_q$
the discrepancies of $( C_1 R)_q(\und x, P,S,I,J)$ and we define
\[
N^{\rm right}_{\le q} = \sum_{i=1}^q \mathbf 1_{\ell_i = \text{right}},\qquad
N^{\rm left}_{\le q} = \sum_{i=1}^q \mathbf 1_{\ell_i = \text{left}}.
\]
For $q>m$ we call $I_q$, $J_q$
the discrepancies of $( C_2 R')_{q-m}( C_1 R)_m(\und x, P,S,I,J)$ and we set
\[
N^{\rm right}_{> q} = \sum_{i=q+1}^{2m} \mathbf 1_{\ell_{i-m} = \text{right}},\qquad
N^{\rm left}_{> q} = \sum_{i=q+1}^{2m} \mathbf 1_{\ell_{i-m} = \text{left}}.
\]
 We  prove below that
  \begin{equation}
  \label{4.13}
N^{\rm right}_{\le q} - |I_{ q}| = N^{\rm left}_{\le q} - |J_{ q}| \ge 0,\quad q\le m
   \end{equation}
     \begin{equation}
  \label{4.14}
N^{\rm right}_{> q} - |I_{ q}| = N^{\rm left}_{> q} - |J_{q}|\ge 0,\quad q> m
   \end{equation}
and observe that if we put $q=2m$ in \eqref{4.14} we get $I_{2m}= J_{2m} = \emptyset$,
so that the theorem follows from \eqref{4.13}--\eqref{4.14}.

{\it Proof of \eqref{4.13}}.   \eqref{4.13} trivially holds
for $q=0$ so that proceeding by induction we
suppose that \eqref{4.13} holds
with $q-1 <m$. Take for instance   $\ell_q = $ left.  Then
 $N^{\rm left}_{\le q} =
N^{\rm left}_{\le q-1}+1$ while $N^{\rm right}_{\le q} =
N^{\rm right}_{\le q-1}$.  Recalling  the definition of $C_1^{\rm left}$, in case (a) or (b) $|J_q|=|J_{q-1}|+1$, and $|I_q|=|I_{q-1}|$; while  in case (c)
$|J_q|=|J_{q-1}|$ and $|I_q|=|I_{q-1}|-1$, thus in all cases
 \eqref{4.13} holds with $q$.  The case when $\ell_q = $ right is analogous and omitted.

{\it Proof of \eqref{4.14}}. As before we proceed by induction observing first that  \eqref{4.14} holds for $q=m$. In fact by definition $N^{\ell}_{> m}=N^{\ell}_{\le m}$  for $\ell=$ right and left. We then assume
\eqref{4.14} holds for $q-1\in(m,2m)$. Suppose for instance that $\ell_q = $ left.  Then
 $N^{\rm left}_{> q} =
N^{\rm left}_{> q-1}-1$ while $N^{\rm right}_{> q} =
N^{\rm right}_{> q-1}$. Recalling  the definition of $C_2^{\rm left}$, in case (a) or (b) if $J_{q-1}\ne \emptyset$ then $|J_q|=|J_{q-1}|-1$, and $|I_q|=|I_{q-1}|$;  if instead $J_{q-1}= \emptyset$ then $|J_q|=|J_{q-1}|$ and $|I_q|=|I_{q-1}|+1$. In  case (c)
$|J_q|=|J_{q-1}|-1$ and $|I_q|=|I_{q-1}|$, thus in all cases
 \eqref{4.14} holds with $q$.  The case when $\ell_q = $ right is analogous and omitted.     \end{proof}

\noindent {\bf Proof of Theorem \ref{teo3.1}}. Given $\om$, $\si$ and $\si'$ as in the statement of Theorem \ref{teo3.1}, we use Lemma \ref{lem4.1} to construct a splitting $(P,S,I,J)$ such that $I=J=\emptyset$. Let $m$ be such that $s_m\le \eps^{-2}\delta$ and $s_{m+1}> \eps^{-2}\delta$. 

{\it Proof of  \eqref{3.2}}.  For 
$q=1,\dots,m$ let $R_q$  be the maps corresponding to the times intervals $(s_q,s_{q+1})$  and let $R'_{1}$ be the map corresponding to the time interval $(s_m,\eps^{-2}\delta)$. Furthermore let  $R'_q$=identity for all $q=2,\dots,m$. Then \eqref{4.12}
 is a splitting of $\big(\und x(t), \si(\eps^{-2}\delta),{\si'}^{(\delta,-)}(\eps^{-2}\delta)\big)$. From Theorem \ref{teo4.2}  we then have that $I^*=J^*=\emptyset$ and thus by Lemma \ref{lem4.1} we get \eqref{3.2}.
 
{\it Proof of  \eqref{3.3}}.  We let $R_q$=identity for all $q=1,\dots,m$ and instead, for $q=1,\dots,m$,  $R'_q$  are the maps corresponding to the times intervals $(s_q,s_{q+1})$. Finally $R'_{m+1}$ is the map corresponding to the time interval $(s_m,\eps^{-2}\delta)$. Then 
\begin{equation}
 \nn
(\und x^*, P^*,S^*,I^*,J^*):=R'_{m+1}( C_2 R')_mR_{m+1}( C_1 R)_m(\und x, P,S,I,J)
   \end{equation}
 is a splitting of $\big(\und x(t),\si^{(\delta,+)}(\eps^{-2}\delta), \si'(\eps^{-2}\delta)\big)$.  Since $R'_{m+1}$ does not change the sets of discrepancies, from Theorem \ref{teo4.2}  we get that $I^*=J^*=\emptyset$ which, by Lemma \ref{lem4.1} concludes the proof of \eqref{3.3}.
\qed

\section{Proof of Theorem \ref{teo1}.}
\label{sec.3}

For any $\eps>0$ we choose an initial configuration $(\und x,\und\si)$ 
with law $P^\eps$ (as described in Section \ref{sec1}) 
and study its evolution
$(\und x(t),\und\si(t))$ for a fixed time interval $[0,T]$.  We do not have
a good knowledge of $(\und x(t),\und\si(t))$ (just that the process is well defined).
The information needed to prove  Theorem \ref{teo1} will be gained by studying two
auxiliary processes $(\und x(t),\und \si^{(\delta,\pm)}(t))$
(which start at time 0 from $(\und x,\und\si)$ as the true process) and by using the inequalities
of the previous section
to compare the true and the auxiliary processes.

Thus the first step is to
extend  the definition  of the auxiliary processes
to the whole time interval $[0,T]$.  This is done in Definition \ref{de4.1} below
by iterating the definition given in the last section 
to the intervals  $[(k-1)\eps^{-2}\delta,k\eps^{-2}\delta]$,
$k \le K$, $K$  the smallest
integer such that $K\eps^{-2}\delta \ge T$.
To this purpose we consider the set $\mathcal X_{K\eps^{-2}\delta }$ defined in the previous section (see Definition \ref{de4.0}) and we prove below that with large probability we can restrict our analysis to trajectories in $\mathcal X_{K\eps^{-2}\delta }$.
\begin{lemma}
\label{lem4.1bis}
There is a  positive constant $c$
independent of $\eps$ (but it may depend on  $\delta$ and $T$) such that
	\begin{equation}
	\label{e1}
\mathcal P^\eps\big[\mathcal X_{K\eps^{-2}\delta} \big ]\ge 1 - e^{-c\eps^{-1}}
    \end{equation}
   where $\mathcal P^\eps$ is defined in Definition \ref{def3.0}.
\end{lemma}

\begin{proof}
Call $Z=\int (u+v)$ and
$p_a:= \frac 1Z\int u \in (0,1)$.
By \eqref{4.17} and \eqref{4.18},
\[
P^\eps [\si_i = a]= \frac1{ Z_\eps}\sum_x u(\eps x),\qquad Z_\eps=\sum_x [u(\eps x)+v(\eps x)]
\]
and since  the $\si_i$ are independent variables, given $\zeta>0$ such that
$ 0< p_a-\zeta<p_a+\zeta<1$ we have for $\eps>0$ small enough
\[
P^\eps [|h_a(\si)- \eps^{-1} p_a|< \zeta ]\ge  1 - e^{-c\eps^{-1}},
\]
with $c$ a suitable positive constant.  Recalling \eqref{3.4}, the number $N_a(t)$ of $a$ particles
at time $t$ is a nearest neighbor symmetric random walk with jump  intensity $2\eps \kappa$, until the time
when $N_a(t)$ reaches either 0 or $M$.  Thus
 \[
\mathcal P^\eps [N_a(t) \in (0,M) \:\text{for all $t\le K\eps^{-2}\delta $} ]\ge  1 - e^{-c\eps^{-1}}
\]with $c$ a new suitable constant.
\end{proof}

\begin{definition}
\label{de4.1}
Chose an initial configuration $(\und x,\und \si)$ as above, fix a $(\und \si, \om)\in \mathcal X_{K\eps^{-2}\delta}$ and a trajectory $\und x(t)$, $t\le K\eps^{-2}\delta$. We call $m_k$, $k=0,\dots,K$ the positive integers such that $k\eps^{-2}\delta\le s_{m_k+1}<s_{m_k+2}\dots$$<s_{m_{k+1}}$. We also call $t_k=k\eps^{-2}\delta
$. We then define $\und \si^{(\delta,+)}(t)$ as the function left continuous with right limits obtained by setting $\und \si^{(\delta,+)}(t)=\und \si^{(\delta,+)}(t_k+)$ for $t\in(t_k,t_{k+1}]$ and 
	\begin{equation}
	\nn
\big(\und x( t_k+),\und \si^{(\delta,+)}(t_k+)\big) =  \prod_{i=m_k+1}^{m_{k+1}} H^{\ell_i}\big(\und x(t_k),\und \si^{(\delta,+)}(t_k)\big),\qquad t_k=k\eps^{-2}\delta.
  \end{equation}
 Similarly $\und \si^{(\delta,-)}(t)$ is the function right continuous with left limits obtained by setting  $\und \si^{(\delta,-)}(t)=\und \si^{(\delta,-)}(t_k)$ for all $t\in[t_k,t_{k+1})$, while at $t_{k+1}=(k+1)\eps^{-2}\delta$ 
 	\begin{equation}
	\nn
	\big(\und x(t_{k+1}),\und \si^{(\delta,-)}(t_{k+1})\big) =
 \prod_{i=m_k+1}^{m_{k+1}}  H^{\ell_i}\big(\und x(t_{k+1}),\und \si^{(\delta,-)}(t_{k}) \big).
  \end{equation}

\end{definition}
An immediate corollary of Theorem \ref{teo3.1} is

\begin{corollary}
In $\mathcal X_{K\eps^{-2}\delta} $ setting $t_k=k\eps^{-2}\delta
$ we have for all $k \le K$
	\begin{equation}
	\label{e2}
 \big(\und x(t_k), {\si}^{(\delta,-)}(t_k)\big)
 \preccurlyeq \big(\und x(t_k), \si(t_k)\big)
 \preccurlyeq \big(\und x(t_k), \si^{(\delta,+)}(t_k)\big),
	\end{equation}
where all the above evolutions start from the same initial datum
$(\und x, \und \si)$.

\end{corollary}

\begin{proof} The number  $N_a(k\eps^{-2}\delta)$ of $a$ particles at time $k\eps^{-2}\delta$ is the same in all the three evolutions. This is evidently true for $k=0$ because they all start from the same configuration and the claim follows because
$$N_a((k+1)\eps^{-2}\delta)-N_a(k\eps^{-2}\delta)=\sum^{m_{k+1}}_{i=m_k+1} \Big( \mathbf 1_{\ell_i=\text{right}} -
\mathbf 1_{\ell_i=\text{left}}\Big).
$$
The corollary then follows from Theorem
\ref{teo3.1}.
\end{proof}

Next step is to prove that $(\und x(k\eps^{-2}\delta), {\si}^{(\delta,\pm)}(k\eps^{-2}\delta))$ have a limit as $\eps\to 0$.  The limit
will be described by the following macroscopic evolutions:

\begin{definition}
\label{def5.1}
For $u,v\in L_1(\mathbb R,\mathbb R_+)$ and $\delta>0$ let $R_\delta(u)$ and $D_\delta(v)$ be such that
	\begin{eqnarray}
\int_{R_\delta(u)}^\infty u(r)dr=\kappa\delta,\qquad \int^{D_\delta(v)}_{-\infty} v(r)dr=\kappa\delta.
	\end{eqnarray}
We define $K^{(\delta)}(u,v)=(u',v')$ with
   \begin{eqnarray}
&&\nn
 u' (r)=
\mathbf 1_{(-\infty,R_\delta(u)]}(r)u(r)+\mathbf 1_{(-\infty,D_\delta(v)]}(r)v(r),
 \\\label{4.2k}
  \\&&\nn
 v' (r)= \mathbf 1_{[D_\delta(v),+\infty)}(r)v(r)+
\mathbf 1_{[R_\delta(u),+\infty)}(r)u(r).
	 \end{eqnarray}
Denote by  $G_t\star u$ the convolution of the Gaussian kernel with a function $u$:
	\begin{equation}
	\label{5.3}
G_t(r,r') =  \frac{e^{-(r-r')^2/2t}}{\sqrt{2\pi t}},\qquad G_t\star u=\int G_t(r,r')u(r')dr'.
\end{equation}
With an abuse of notation we write $G_t \star (u,v)\equiv (G_t\star u,G_t\star v)$.
We define
the ``barriers''  $S_{n\delta}^{(\delta,\pm)}(u,v) $,  $n\in \mathbb N$,
by setting
$S^{(\delta,\pm)}_{0}(u,v)=(u,v)$, and $\forall n\ge1$
\begin{eqnarray}
\nn
&&
S^{(\delta,-)}_{n\delta}(u,v)= K^{(\delta)} G_\delta*S^{(\delta,-)}_{(n-1)\delta}(u,v),
\\&&
S^{(\delta,+)}_{n\delta}(u,v)= G_\delta * K^{(\delta)} S^{(\delta,{+})}_{(n-1)\delta}(u,v).
\label{5.4}
\end{eqnarray}
We denote by $(u^{(\delta,\pm)}_{n\delta},v^{(\delta,-)}_{n\delta}) = S^{(\delta,\pm)}_{n\delta}(u,v)$.
\end{definition}

\begin{thm}
\label{thme1}

 For any $k \le K$ and any $\delta$ small enough
 \[
\eps  \xi_{(\und x(k\eps^{-2}\delta), \und \si^{(\delta,\pm)}(k\eps^{-2}\delta))} \to S_{k\delta}^{(\delta,\pm)}(u,v)
 \]
as $\eps\to 0$ weakly in probability.

\end{thm}

The auxiliary processes are essentially independent random walk evolutions
with an additional colors change at finitely many times, $k\eps^{-2}\delta$, $k\le K$.
The convergence of the random walk evolutions can be established in a very strong form
which allows to control the positions of the rightmost $a$ and leftmost $b$ particles.
The argument is rather lengthy but essentially analogous to that in \cite{CDGP1} and for
brevity we omit it.

\begin{thm}
\label{teor4.4}
There exist continuous functions $\bar u(r,t),\bar v(r,t)$, $r\in \mathbb R$, $t\in [0,T)$, also
denoted by $(\bar u(r,t),\bar v(r,t))=S_t(u,v)$ such that $S_0(u,v)=(u,v)$ and for any $t\in [0,T)$:
  \begin{equation}
  	\label{9.1}
\lim_{n\to \infty}  S_{2^{-n}t}^{(\delta,\pm)}(u,v) =S_t(u,v),
	\end{equation}
uniformly in the compacts and in $L^1$.

\end{thm}

We refer to  Section 8 of \cite{CDGP1} where an analogous statement has been proved.
Fix $t$, by \eqref{e2}  and Theorem \ref{thme1} with  $\delta=2^{-n}t$, for any $r\in \mathbb R$, in probability
	\begin{equation}
	\label{e3}
\limsup_{\eps\to 0}\eps\sum_{y \ge \eps^{-1}r}\xi_{( \und x( \eps^{-2}2^{-n}t), \und \si( \eps^{-2}2^{-n}t))}(y) \le \int_{r}^{+\infty} u^{(2^{-n}t,+)}_t,
	\end{equation}
	\begin{equation}
	\label{e4}
\int_{r}^{+\infty} u^{(2^{-n}t,-)}_t \le \liminf_{\eps\to 0}\eps\sum_{y \ge \eps^{-1}r}\xi_{( \und x( \eps^{-2}2^{-n}t), \und \si( \eps^{-2}2^{-n}t))}(y).
	\end{equation}
Theorem \ref{teo1} then follows because by \eqref{9.1}, the integrals in \eqref{e3}--\eqref{e4} converge   as $n\to \infty$ to the same limit
$\dis{ \int_{r}^{+\infty} \bar u(r',t)dr'}$. Details are omitted.

\section{Macroscopic inequalities.}
	\label{sec5}
	
In this section
we assume that for some $S>0$ there exists a solution 
$(\upmu(\cdot,t),U_t)$, $(\upnu(\cdot,t),V_t)$, $t\in[ 0,S]$
of the free boundary problem \eqref{eq11}. We assume that this solution is regular in the sense specified before Theorem \ref{teo2}.

 The main result of this section is Theorem \ref{thm5.1} below that states that, modulo an error exponentially small in $\delta$, 
 $(\upmu(\cdot,t),\upnu(\cdot,t))$ is in between the barriers $S^{(\delta,\pm)}_{n\delta}\big(\upmu_0,\upnu_0\big)\equiv (u^{(\delta,\pm)}_{n\delta},v^{(\delta,\pm)}_{n\delta})$, $\upmu_0=\upmu(\cdot,0)$,$\upnu_0=\upnu(\cdot,0)$. The inequalities are the macroscopic analogue of the microscopic ones.

	\begin{thm}
\label{thm5.1}
There is $\delta_0$ so that  the following holds. There are constants $c$ and $c'$ so that for all $\delta<\delta_0$, for all  $k\le \delta^{-1}S$
and for all $r\in \mathbb R$ we have
\begin{equation}
F(r;u_{k\delta}^{(\delta,-)})-kc'e^{- c\delta^{-1}}\le  F(r;\upmu(\cdot, k\delta))\le
F(r;u_{k\delta}^{(\delta,+)})+kc'e^{- c\delta^{-1}},
\label{eq5.20}
	\end{equation}
where 	$F(r;g)=\int_r^{+\infty} g$.
	\end{thm}
We first prove Theorem \ref{teo2} as a corollary of  Theorem \ref{thm5.1}.

\noindent {\bf Proof of  Theorem \ref{teo2}}. Fix a $t\le S$ and consider $k=$ integer part of $\delta^{-1}t$,  then take the limit $\delta\to 0$ in \eqref{eq5.20}  using Theorem \ref{teor4.4} we then get that
$(\mu(\cdot,t),\nu(\cdot,t))$ coincide with $(\bar u(\cdot,t),\bar v(\cdot,t))$ of Theorem \ref{teo1}.\qed

We  prove
  in Subsection \ref{subsub521} the lower bound  and in Subsection \ref{subsub523} the upper bound  in \eqref{eq5.20}  for $k=1$, finally, in  Subsection \ref{sub5.2} we prove that we can reduce the generic step to this case. We first need to state properties of the regular solutions that will be used in the sequel.

\subsection{Properties of a regular solution.}
\label{sub5.1}

 The regular solution $(\upmu(\cdot,t),U_t)$, $(\upnu(\cdot,t),V_t)$, $t\in[ 0,S]$ is related to the law $P_{r',s}$ of a Brownian motion $\{B_t,t\ge s\}$ that starts from $r'\in \mathbb R$ at time $s\in [0,S]$ in the following way, see for instance \cite{KS}. First define the stopping times
	\begin{equation}
	\label{5.6a}
\tau_s^{U}=\inf\Big\{t\ge s : B_t\ge U_t\Big\},\qquad
\tau_s^{V}=\inf\Big\{t\ge s : B_t\ge V_t\Big\}.
	\end{equation}
Then
for  any $t\in [0,S]$ and any  interval $I\subset \mathbb R$
	\begin{equation}
\int_I \upmu(r,t)dr =\int \upmu_0(r') P_{r',0}(B_t\in I; \tau_0^{U}>t)
+\kappa\int_0^t P_{V_s,s}(B_t\in I; \tau_s^{U}>t),
 \label{5.2a} 	\end{equation}
	\begin{equation}
\int_I  \upnu(r,t)dr =\int \upnu_0(r') P_{r',0}(B_t\in I; \tau_0^{V}>t)
+\kappa\int_0^t P_{U_s,s}(B_t\in I; \tau_s^{V}>t).
 \label{5.2b}
	\end{equation}

We call $P_{r,s}(\tau^U_s\in dt)$ and $P_{r,s}(\tau^V_s\in dt)$ the  law of the stopping times  \eqref{5.6a}.

	\begin{lemma}
 \label{lemma5.1}
For all $t\in[0,S]$ we have
	\begin{eqnarray}
	\label{5.12}
&&	\hskip-1cm \int \upmu_0(r) P_{r,0}(\tau^U_0\le t)dr +\kappa \int_0^t  P_{V_s,s}(\tau^U_s\le t) ds=\kappa t,
\\&&		\label{5.12b}
	\hskip-1cm  \int \upnu_0(r) P_{r,0}(\tau^V_0\le t)dr +\kappa \int_0^t  P_{U_s,s}(\tau^V_s\le t) ds=\kappa t.
	 	\end{eqnarray}
Moreover, there are $C$ and $C'$ depending on the constant $c>U_t-V_t$ such that for  all $\delta$ small enough the following holds. For all $r^*\in \mathbb R$ and $t\le \delta$
\begin{eqnarray}
	\nn&&
\hskip-1cm
\Big| \kappa\int P_{V_s,s}(B_{t-s}\ge r^*; \tau_s^{U}>t)dr
\\&&-\int_0^t \int \upnu_0(r) P_{r,0}(\tau^V_0\in ds) P_{V_s,s}(B_{t}\ge r^*; \tau_s^{U}>t)dr  \Big|\le C' e^{- C\delta^{-1}},\nn \\
\label{5.14}
\end{eqnarray}	
\begin{eqnarray}
	\nn
&&\hskip-1cm
 \Big| \kappa\int P_{U_s,s}(B_{t-s}\le r^*; \tau_s^{V}>t)dr
\\&&-\int_0^t \int \upmu_0(r) P_{r,0}(\tau^V_0\in ds) P_{U_s,s}(B_{t}\le r^*; \tau_s^{V}>t)dr ds\Big|\le C' e^{- C\delta^{-1}}.\nn \\
\label{5.15}
\end{eqnarray}	

	\end{lemma}

\begin{proof}  From \eqref{5.2a} we have
   	\begin{eqnarray*}
&&\hskip-1cm
 \int \upmu(r,t)dr=\int \upmu_0(r) P_{r,0}(\tau^U_0>t)dr +\kappa\int_0^t  P_{V_s,s}(\tau^U_s>t) ds
\\&&=\int \upmu_0(r)dr+\kappa t-
	 \int \upmu_0(r) P_{r,0}(\tau^U_0\le t)dr -j \int_0^t  P_{V_s,s}(\tau^U_s\le t) ds.
	 	\end{eqnarray*}
	Since the total mass is conserved this yields \eqref{5.12}.
The proof of \eqref{5.12b} is analogous.
Differentiating equations \eqref{5.12} and \eqref{5.12b} and noticing that $P_{r,0}(\tau_0^U\in dt)$ is absolutely continuous with respect to the Lebesgue measure, we get 
\begin{equation}
\kappa=\int \upmu_0(r) P_{r,0}(\tau_0^U\in dt)dr
+\kappa\int_0^t P_{V_{s},s}(\tau_s^U\in dt)ds,
 \label{5.3a}
\end{equation}
	\begin{equation}
	 \label{5.3b}
\kappa=\int \upnu_0(r) P_{r, 0}(\tau_0^V\in dt)dr
+\kappa\int_0^t P_{U_s,s}(\tau_s^V\in dt)ds.
\end{equation}
	We now use \eqref{5.3b}  to rewrite $\kappa$ on the right hand side of \eqref{5.2a} 
	as follows
		\begin{eqnarray}
		\nn &&
\hskip-1.7cm  \kappa\int_0^t P_{V_s,s}(B_t\ge r^*; \tau_s^{U}>t)ds\\&&\nn
 \hskip.4cm =\int_0^t \int v_0(r) P_{r, 0}(\tau^V_0\in ds) P_{V_s,s}(B_t\ge r^*; \tau_s^{U}>t)dr
\nn\\&& \hskip.4cm +
\int_0^t  \int_0^s \kappa P_{U_{s'},s'}(\tau^V_{s'}\in ds) P_{V_s,s}(B_t\ge r^*; \tau_s^{U}>t)ds'.
 \label{5.11}
	\end{eqnarray}
There are $C,C'>0$ so that for all $0\le s'<s<\delta$
	\begin{eqnarray}
P_{U_{s'},s'}(\tau_{s'}^V<s)\le C'e^{-C\delta^{-1}},\qquad
P_{V_{s'},s'}(\tau_{s'}^U<s)\le C'e^{-C\delta^{-1}}.
	\label{5.111}
	\end{eqnarray}
	To prove \eqref{5.14} we observe that the last term in \eqref{5.11} is  bounded by \eqref{5.111}.
The proof of \eqref{5.15} is analogous by using \eqref{5.2b}  and \eqref{5.3a}.
 \end{proof}

\subsection{Lower bound in the first time interval.}
\label{subsub521}

Here we prove the first inequality in \eqref{eq5.20}  for $k=1$ observing that in the proof we only use that the evolution $S^{(\delta,-)}_\delta(\upmu_0,\upnu_0)$ has same initial datum as the regular solution. 
More precisely we prove that for all $r^*\in\mathbb R$
	\begin{equation}
	\label{5.33}
F(r^*;\upmu(\cdot,\delta))=\int_{r^*}^\infty \upmu(r,\delta)dr \ge \int_{r^*}^\infty u_\delta^{(\delta,-)}(r)dr -3C'e^{-C\delta^{-1}},
	\end{equation}
with $C'$ and $C$ as in Lemma \ref{lemma5.1}.

By definition $\dis{u^{(\delta,-)}_\delta=1_{(-\infty,R)}
G_\delta\star  \mu_0+ \mathbf 1_{(-\infty,D]}G_\delta\star  \nu_0}$ with $R$, $D$ so that
	\begin{equation}
	\label{5.30}
\int_{R}^\infty G_\delta\star \upmu_0=\kappa\delta,\qquad
\int^{ D}_{-\infty} G_\delta\star \upnu_0=\kappa\delta.
	\end{equation}
By using the law of the Brownian motion we write
		\begin{eqnarray}	
&&\hskip-1cm	\int_{r^*}^\infty u_\delta^{(\delta,-)}= \int \upmu_0(r) P_{r,0}\big(B_\delta\in (r^*,  R)\big)dr
+ \int \upnu_0(r) P_{r,0}\big(B_\delta\in [r^*,   D)\big)dr	
\nn \\&&=\int \upmu_0(r)P_{r,0}(B_\delta\ge r^*)-\kappa\delta
+ \int \upnu_0(r) P_{r,0}\big(B_\delta\in [r^*, D)\big)dr	.
\label{5.31}
\end{eqnarray}
Using \eqref{5.2a} and \eqref{5.12} we get
	\begin{equation}
 \int_{r^*}^\infty \upmu(r,\delta)dr
\ge
\int \upmu_0(r)P_{r,0}(B_\delta\ge r^*)+\kappa\int_0^\delta P_{V_s,s}(B_{\delta-s}\ge r^*)-\kappa\delta.
\label{5.17a}
\end{equation}	
Thus if $r^*> D$ from \eqref{5.31} and \eqref{5.17a} we get
\eqref{5.33}.
We then assume that $r^*\le D$  and
observe that
 by \eqref{5.14}  and \eqref{5.111}
		 \begin{eqnarray}
		 &&\nn
\hskip-2.5cm
\kappa\int_0^\delta P_{V_s,s}(B_{\delta-s}\ge r^*)\ge
\kappa\int_0^\delta P_{V_s,s}(B_{\delta-s}\ge r^*; \tau_s^{U}>\delta)\\&& \ge \int  \upnu_0(r) P_{r,0}(B_\delta\ge r^*;\tau_0^V\le \delta)dr
- 2C' e^{- C\delta^{-1}}.
 \label{5.35}
		\end{eqnarray}
By  \eqref{5.12b} and \eqref{5.111}, $\dis{\int \upnu_0(r) P_{r,0}(\tau_0^V\le \delta)dr\ge \kappa\delta- C'e^{-C\delta^{-1}}}$.
Thus
	\begin{eqnarray}
	\label{5.40}
&&\hskip-1cm\nn
\int \upnu_0(r) P_{r,0}(B_\delta\ge r^*;\tau_0^V\le \delta)\ge \kappa\delta- \int \upnu_0(r) P_{r,0}(B_\delta\le r^*;\tau_0^V\le \delta)
- C'e^{-C\delta^{-1}}
\\&&\hskip1cm \ge  \kappa\delta- \int \upnu_0(r) P_{r,0}(B_\delta\le r^*)dr
- C'e^{-C\delta^{-1}}.
	\end{eqnarray}
Then from \eqref{5.35}, \eqref{5.40} and the definition of $ D$ we get
  \begin{eqnarray}
  \nn
&&\hskip-1cm  \kappa\int_0^\delta P_{V_s,s}(B_{\delta-s}\ge r^*) \ge \kappa\delta- \int \upnu_0(r)P_{r,0}\big(B_\delta\le r^*\big)dr
- 3C'e^{-C\delta^{-1}},
\\&&\hskip1cm=
\int \upnu_0(r)P_{r,0}\big(B_\delta\in [r^*,  D]\big)dr
- 3C'e^{-C\delta^{-1}},
\label{5.36}
\end{eqnarray}
 concluding the proof of \eqref{5.33}.

\subsection{Upper bound in the first time interval.}
\label{subsub523}

Here we give the proof of the upper bound in \eqref{eq5.20}   for $k=1$.
Call $ R_0$ and $ D_0$ the points such that
	\begin{equation}
	\label{5.17}
\int_{ R_0}^\infty \upmu_0(r)dr=\kappa\delta,\qquad
\int^{D_0}_{-\infty} \upnu_0(r)dr=\kappa\delta.
	\end{equation}
and call $u_2=\upmu_0-u_1$, $v_2=\upnu_0-v_1$ where
	\begin{eqnarray}
u_1(r)= \upmu_0(r)\mathbf 1_{( R_0,+\infty)}(r)
\qquad v_1(r)= \upnu_0(r))\mathbf 1_{(-\infty,D_0]}(r).
	\label{a5.26}
	\end{eqnarray}
	Thus $v_1$ and $u_1$ have mass $\kappa\delta$ and  by definition
	\begin{equation*}	
	u^{(\delta,+)}_\delta=G_\delta\star [u_2+v_1],\qquad 	v^{(\delta,+)}_\delta=G_\delta\star [v_2+u_1].
		\end{equation*}
From  \eqref{5.2a} we get that the inequality $F(r^*; \mu(\cdot,\delta))\le F(r^*; u^{(\delta,+)}_\delta)+m$ can be written as
	\begin{eqnarray}
	\nn
&&\hskip-1cm
\int \upmu_0(r)P_{r,0}(B_\delta\ge r^* ; \tau_0^{U}>\delta)dr
+\kappa\int_0^\delta P_{V_s,s}(B_{\delta-s}\ge r^*; \tau_s^{U}>\delta)ds
\\&&\hskip2cm \le
\int [u_2(r)+v_1(r)] P_{r,0}(B_\delta \ge r^*)dr + m.
\label{5.23k1}
\end{eqnarray}	
We prove below \eqref{5.23k1} for $m=4C'e^{-C\delta^{-1}}$ with $C'$ and $C$ as in Lemma \ref{lemma5.1}.
Since $\upmu_0=u_1+u_2$ we have
\begin{eqnarray}
	\nn
&&\hskip-.7cm
\int \upmu_0(r)P_{r,0}(B_\delta\ge r^* ; \tau_0^{U}>\delta)dr
=\int u_2(r) P_{r,0}(B_\delta\ge r^*)dr
\\&&\hskip.6cm
+\int u_1(r)P_{r,0}(B_\delta\ge r^* ; \tau_0^{U}>\delta)dr
- \int u_2(r) P_{r,0}(B_\delta\ge r^* ; \tau_0^{U}\le \delta)dr.
\nn\\\label{a5.23}
\end{eqnarray}	
From \eqref{5.14},  \eqref{5.111} and using that $\nu_0=v_1+v_2$ we have	\begin{eqnarray}
&&\hskip-.6cm \nn
\kappa\int_0^\delta P_{V_s,s}(B_{\delta-s}\ge r^*; \tau_s^{U}>\delta)\le \int \nu_0(r) P_{r,0}(B_\delta\ge r^*;\tau^D\le \delta)dr
\\&&\hskip5cm\nn +2C'e^{-C\delta^{-1}}
\\&&\hskip.5cm\nn
= \int v_2(r) P_{r,0}(B_\delta\ge r^*;\tau^V\le \delta)dr
- \int v_1(r) P_{r,0}(B_\delta\ge r^*;\tau^V>\delta)dr
\\&&\hskip1.5cm+\int v_1(r) P_{r,0}(B_\delta\ge r^*)dr+2C'e^{-C\delta^{-1}}.
	\label{5.21}
	\end{eqnarray}	
From 
 \eqref{a5.23} and \eqref{5.21} we get
	\begin{eqnarray}
	\nn
&&\hskip-.7cm	F(r^*; \mu(\cdot,\delta))\le \int [u_2(r)+v_1(r)] P_{r,0}(B_\delta \ge r^*)dr + 2C'e^{-C\delta^{-1}}
	\\&&\hskip.3cm \nn
+\int u_1(r)P_{r,0}(B_\delta\ge r^* ; \tau_0^{U}>\delta)dr
- \int u_2(r) P_{r,0}(B_\delta\ge r^* ; \tau_0^{U}\le \delta)dr
\\&&\hskip.3cm\nn\nn
+ \int v_2(r) P_{r,0}(B_\delta\ge r^*;\tau^V\le \delta)dr
- \int v_1(r) P_{r,0}(B_\delta\ge r^*;\tau^V>\delta)dr. \nn
\\\label{a5.30}
		\end{eqnarray}
In Lemma \ref{le54} below we prove that the last two terms on the right hand side of \eqref{a5.30} are bounded by $C'e^{-C\delta^{-1}}
$ thus concluding the proof of  \eqref{5.23k1}.

	\begin{lemma}
	\label{le54}
Let $u_i$ and $v_i$, $i=1,2$ be as in \eqref{a5.26}, then for all $r^\star\in\mathbb R$
	\begin{equation}
\int u_1(r)P_{r,0}(B_\delta\ge r^* ; \tau_0^{U}>\delta)\le
\int u_2(r) P_{r,0}(B_\delta \ge r^* ; \tau_0^{U}\le \delta)+C'e^{-C\delta^{-1}},
	\label{5.22}\end{equation}
		\begin{equation}
\int v_1(r)P_{r,0}(B_\delta\ge r^* ; \tau_0^V>\delta)\ge
\int v_2(r) P_{r,0}(B_\delta \ge r^* ; \tau_0^V\le \delta)-C'e^{-C\delta^{-1}}.
	\label{5.23}
	\end{equation}
	\end{lemma}
	
\begin{proof}
	We only prove \eqref{5.22}  since the proof \eqref{5.23} is completely analogous. 	
	From   \eqref{5.12} we get
	\begin{eqnarray}
	\nn
 \int [u_1(r)+u_2(r)] P_{r,0}(\tau^U_0\le \delta)dr +\kappa \int_0^\delta  P_{V_s,s}(\tau^U_s\le \delta) ds=\kappa\delta=\int u_1(r)dr,
	\end{eqnarray}
thus
		\begin{equation}
 \int u_1(r) [1-P_{r,0}(\tau^U_0\le \delta)]dr = \int u_2(r) P_{r,0}(\tau^U_0\le \delta)dr
 +\kappa \int_0^\delta  P_{V_s,s}(\tau^U_s\le \delta).
	 \label{5.25}
		\end{equation}
We call
	\begin{equation}
\alpha(r)=P_{r,0}(\tau^U_0\le \delta),\qquad
 \beta(s)=  P_{V_s,s}(\tau^U_s\le \delta)
	 \label{5.45}
		\end{equation}
and from \eqref{5.25} we get
	\begin{equation}
Z:= \int u_2(r) \alpha(r)
 +\kappa \int_0^\delta  \beta(s)=  \int u_1(r) [1-\alpha(r)]dr.
	 \label{5.46}
		\end{equation}
We call $\la_r(ds)$ the law of $\tau^U_0$ conditioned to the event $\tau^U_0\le \delta$ when the Brownian motion starts from $r$ at time 0  and write
		\begin{equation}
	\int u_2(r) P_{r,0}(B_\delta \ge r^* ; \tau_0^{U}\le \delta)dr=\int u_2(r) \alpha(r)\int_0^\delta \la_r(ds)P_{U_s,s}(B_\delta \ge r^*).
	\label{5.47}
		\end{equation}
We denote  by $\nu_s(ds')$ the law of $\tau^U_s$ conditioned to the event $\tau^U_s\le \delta$ when the Brownian motion starts from $V_s$ at time $s$  and write
	\begin{equation}
	\kappa \int_0^\delta  P_{V_s,s}(B_{\delta-s}\ge r^*;\tau^U_s\le \delta)
	=	\kappa \int_0^\delta  \beta(s)\int_s^\delta
	\nu_s(ds') P_{V_s',s'}(B_{\delta-s'}\ge r^*).
	\label{5.48}
		\end{equation}
From \eqref{5.46}, \eqref{5.47} and \eqref{5.48} it follows that there  exists a non negative measure $g(dt)$ on $[0,\delta]$,
so that $\dis{\int_0^\delta g(dt) =Z}$ and
		\begin{eqnarray}
		\nn
&& \hskip-1cm
\int u_2(r) P_{r,0}(B_\delta \ge r^* ; \tau_0^{U}\le \delta)dr+
	\kappa \int_0^\delta  P_{V_s,s}(B_{\delta-s}\ge r^*;\tau^U_s\le \delta)
\\&&\hskip1cm
	= \int_0^\delta  \, g(dt) P_{ U_t;t}\Big(B_{\delta-t} \ge r^*\Big).
	\label{5.49}
		\end{eqnarray}
Thus since by  \eqref{5.46} the measures $u_1(r)[1-\alpha(r)]dr$ and $g(dt)$ have same mass $Z$
 then,
by the isomorphism of Lebesgue measures, \cite{roklin},  there is a map $\Ga:\mathbb R \to [0,\delta]$ so that
\begin{equation}
\int_0^\delta g(dt) P_{ U_t;t}\Big[B_\delta \ge r^*\Big]  = \int u_1(r)[1-\alpha(r)]
P_{ U_{\Ga(r)};\Ga(r)}\big(B_\delta \ge r^*\big)dr.
\label{7.24a}
	\end{equation}
	We use the following inequality  proved in  \cite{CDGP2})(see the proof of (5.36) in this paper). If $\ga=(\ga(t)$, $t\ge 0)$ is a $C^1$-curve then for all $\delta>0$
\begin{eqnarray}
 P_{r;0}\Big[B_\delta \ge r\,\big|\, {\tau_0^\ga} >\delta\Big] \le
P_{\ga_{t};t}\Big[B_\delta \ge r\Big],\qquad \forall r\le \ga(0), t\in[0,\delta]
\label{5.28}
	\end{eqnarray}
where $\tau_0^\ga$ is the hitting time of the curve $\ga$.

By  \eqref{5.28} and \eqref{5.111}, from \eqref{5.49} and \eqref{7.24a} we get
	\begin{eqnarray}
	\nn
&&
\hskip-.8cm\int u_1(r)P_{r,0}(B_\delta\ge r^* ; \tau_0^U>\delta)
=\int u_1(r)[1-\alpha(r)]
P_{ r,0}(B_\delta \ge r^*|\tau_0^U>\delta)
\\&&\nn \le \int u_1(r)[1-\alpha(r)]
P_{ U_{\Ga(r)};\Ga(r)}\big(B_\delta \ge r^*\big)
\\&&=\int u_2(r) P_{r,0}(B_\delta \ge r^* ; \tau_0^U\le \delta)dr+
	\kappa \int_0^\delta  P_{V_s,s}(B_{\delta-s}\ge r^*;\tau^U_s\le \delta)
	\nn
	\\&&\le \int u_2(r) P_{r,0}(B_\delta \ge r^* ; \tau_0^U\le \delta)dr+
C'e^{C\delta^{-1}}.	\nn
\end{eqnarray}
This concludes the proof of \eqref{5.22}.
\end{proof}

\subsection{Properties of the barriers.}
\label{sub5.2}

The function $w(\cdot,t)=\upmu(\cdot,t)+\upnu(\cdot,t)$ is the solution of the heat equation:
	\begin{equation}
	\label{6.2}
w(r,t)=(G_t\star w_0)(r),\quad  r\in \mathbb R, \, t\ge 0,\qquad w_0=\upmu(\cdot,0)+\upnu(\cdot,0).
	\end{equation}
Observe that not only the total mass $\int w_0=M_{\text{tot}}$ is conserved but also $\int\upmu(r,t)$ $= \int \upmu(r,0)=:M_0$ for all $t$. Given
$\phi\in L_1(\mathbb R,\mathbb R_+)$ we call
	\begin{equation}
\mathcal B(\phi, M_0):=\Big\{(u,v)\in \mathcal U:
u(r)+v(r)=\phi(r),\,\,\forall r\in \mathbb R, \text{ and }
 \int_{\mathbb R} u=M_0\Big\}.
	\label{6.5}
	\end{equation}
Below we will use the above definition with $\phi=w(\cdot,n\delta)$, because from the definitions  it follows that
	\begin{equation}
	\label{6.3}
u^{(\delta,\pm)}_{n\delta}(r)+v^{(\delta,\pm)}_{n\delta}(r)=w(r,n\delta),\qquad \forall r\in \mathbb R,
 \quad \forall n\le \delta^{-1}T
	\end{equation}
and also that for all $n\le \delta^{-1}T$
\begin{equation}
	\label{6.4}
\int_{\mathbb R} u^{(\delta,\pm)}_{n\delta}=\int_{\mathbb R} \upmu(r,0)=M_0
\qquad
\int_{\mathbb R} v^{(\delta,\pm)}_{n\delta}=\int_{\mathbb R} \upnu(r,0)=M_{\text{tot}}-M_0.
	\end{equation}
		\begin{definition}
	\label{def5.1b}
Given  two pairs
$(u',v')$, $(u,v)\in \mathcal B(\phi, M_0)$
and a number $m\ge 0$, we define
	\begin{equation}
(u',v') \prec (u,v) \text{ modulo } m\quad \text{iff }\, \forall r\in \mathbb R: \quad F(r;u')\le  F(r;u)+m.
\label{6.6}
	\end{equation}
If $m=0$  we  say that $(u',v') \preccurlyeq (u,v)$.
	\end{definition}
At the end of this Subsection we will prove that  \eqref{eq5.20}  for all $k\ge 1$ follows from the one step estimates of Subsections \ref{subsub521} and \ref{subsub523}.
We first  prove that the evolutions $S^{(\delta,\pm)}_{\delta}$ preserve the order in the case $m=0$.
	\begin{lemma}
	\label{lem5.1}
Let $(u',v')$, $(u,v)\in \mathcal B(\phi, M_0)$.
 \begin{equation}
\text{if }\quad (u',v') \preccurlyeq (u,v)\quad  \text{ then } \quad S^{(\delta,\pm)}_{\delta}(u',v')
 \preccurlyeq S^{(\delta,\pm)}_{\delta}(u,v).
 	\label{6.7}
	\end{equation}
Moreover $S^{(\delta,\pm)}_{\delta}(u',v')$ and $S^{(\delta,\pm)}_{\delta}(u,v)$ belong to $\mathcal B(G_\delta\star \phi, M_0)$.
		\end{lemma}
\begin{proof}
We first prove that $
K_\delta$ is non decreasing with respect to  $\preccurlyeq$. Calling $(\bar u',\bar v')=K_\delta(u',v')$ and $(\bar u,\bar v)=K_\delta(u,v)$ we have
\begin{eqnarray}
\nn
&&	\bar u'=u'\mathbf 1_{(-\infty,\mathcal R')}+v'\mathbf 1_{(-\infty,\mathcal D')},\qquad 			\bar v'=u'\mathbf 1_{[\mathcal R',+\infty)}+v'\mathbf 1_{(\mathcal D',+\infty)}	
\\&&
\bar u=u\mathbf 1_{(-\infty,\mathcal R)}+v\mathbf 1_{(-\infty,\mathcal D)},\qquad \bar v=u\mathbf 1_{[\mathcal R,+\infty)}+v\mathbf 1_{(\mathcal D,+\infty)},
\label{5.53}
	\end{eqnarray}
where $\mathcal D$,  $\mathcal D'$,  $\mathcal R$ and  $\mathcal R'$ are the points such that
	\begin{eqnarray}
\int_{\mathcal R'}^\infty u'=\kappa\delta= \int^{\mathcal D'}_{-\infty} v',\qquad
\int_{\mathcal R}^\infty u=\kappa\delta= \int^{\mathcal D}_{-\infty} v.
	\label{5.52}
		\end{eqnarray}
Since $(u',v') \preccurlyeq (u,v)$ we have that $\mathcal D\le \mathcal D'\le 	\mathcal R'\le \mathcal R$. Furthermore $K_\delta(u',v')$ and $K_\delta(u,v)$ are both in the set $B(\phi, M)$.
Using this fact we get
	\begin{equation}
	\label{8.6}
\int_{\mathcal D}^{\mathcal D'} [u'+v']+\int^{\mathcal R'}_{\mathcal D'} u'=\int_{\mathcal D}^{\infty} [u'+v']-\int_{\mathbb R} v'=
\int_{\mathcal D}^{\infty} [u+v]-\int_{\mathbb R} v=\int_{\mathcal D}^{R} u.	
	\end{equation}
For $r\le \mathcal D$, from \eqref{8.6} we get
 	\begin{eqnarray}
\nn
	&&\hskip-1.6cm
F(r;\bar u')=\int^{\mathcal D'}_{r} [u'+v']dr
+\int^{\mathcal R'}_{\mathcal D'} u'
=\int^{\mathcal D}_{r} \phi +\int_{\mathcal D}^{R} u
=F(r;\bar u).
\end{eqnarray}
Analogous computations show that $F(r;\bar u')\le F(r;\bar u)$ for $r\le \mathcal D'$.
For $r> \mathcal D'$
		\begin{eqnarray}
&&\hskip-1.6cm
	F(r;\bar u')=
	\int^{+\infty}_{r} u' -\kappa\delta
	\le \int^{\mathcal R}_{r^*} u=F(r;\bar u).
\nn 	
	\end{eqnarray}
 Thus $F(r;\bar u')\le F(r;\bar u)$ for all $r\in\mathbb R$ and this concludes the proof of the monotonicity of $K_\delta$.
Recalling the definitions, to conclude the proof of the Lemma it is enough to show that also the convolution with $G_\delta$ is non decreasing with respect to $\preccurlyeq$. This fact is a simple adaptation of the proof of Lemma 2.6 of \cite{CDGP1} and thus we omit its proof. \end{proof}

The following Proposition, proved  in Appendix \ref{app}, will allow us to reduce the inequalities modulo $m>0$ to the ones with $m=0$.
\begin{proposition}
\label{prop5.3}
There is $m_0>0$ so that for all $m\in(0,m_0)$ the following holds. Let $(u',v')$, $(u,v)\in \mathcal B(\phi, M_0)$ be such that
$(u',v') \prec (u,v)$ modulo $m<M_0$, $m>0$.
\begin{enumerate}
\item \label{1}
There is $(f^\star,g^\star)\in \mathcal B(\phi, M)$ such that $ (u,v)\preccurlyeq (f^\star,g^\star)$,  $(u',v') \preccurlyeq (f^\star,g^\star) $ and 	\begin{equation}
	\label{5.56}
S^{(\delta,+)}_\delta(f^\star,g^\star)\prec S^{(\delta,+)}_\delta(u,v)\text{ modulo }2m.
		\end{equation}
		\item \label{2}
There is $(f_\star,g_\star)\in \mathcal B(\phi, M)$ so that $(f_\star,g_\star)\preccurlyeq  (u',v')$, $ (f_\star,g_\star) \preccurlyeq (u,v) $ and
	\begin{equation}
	\label{6.10}
S^{(\delta,-)}_\delta(u',v')\prec S^{(\delta,-)}_\delta(f_\star,g_\star)\text{ modulo }2m.
		\end{equation}
\end{enumerate}

\end{proposition}
As a consequence of the above Proposition we now prove \eqref{eq5.20}. We will use the following notation:	\begin{equation}
	\label{8.9}
(\upmu(\cdot,(k+1)\delta),\upnu(\cdot,(k+1)\delta))=T_\delta(\upmu(\cdot,k\delta),\upnu(\cdot,k\delta)).
	\end{equation}
{\bf  Proof of Theorem \ref{thm5.1}}
As a consequence of the estimates in Subsections \ref{subsub521} and  \ref{subsub523} we have that for all $k$, letting $(\hat u,\hat v):=T_{k\delta}(\upmu_0,\upnu_0)$ 
	\begin{eqnarray}
S^{(\delta,-)}_\delta(\hat u,\hat v) \prec T_\delta(\hat u,\hat v) \prec S^{(\delta,+)}_\delta(\hat u,\hat v) \text{ modulo } m:=\bar c e^{-C\delta^{-1}}
 \label{6.14}
 	\end{eqnarray}
with $\bar c=4C'$ and $C$ and $C'$ as in Lemma \ref{lemma5.1}.

Observing that \eqref{6.2}, \eqref{6.3} and \eqref{6.4} imply that for all $k$, $S^{(\delta,\pm)}_{k\delta}(\upmu_0, \upnu_0)$ and $T_{k\delta}(\upmu_0, \upnu_0)$ belong to $\mathcal B(w_{k\delta},M_0)$,
by  \eqref{6.14} with $k=0$ we can use  \ref{1} of Proposition \ref{prop5.3} with $\phi=w(\cdot,\delta)$,  $(u',v')=T_\delta(\upmu_0, \upnu_0)$ and $(u,v)=S^{(\delta,+)}_\delta(\upmu_0, \upnu_0)$. Thus from Lemma
\ref{lem5.1} and \eqref{5.56} we get
	\begin{equation}
S^{(\delta,+)}_\delta (u',v')\preccurlyeq S^{(\delta,+)}_\delta(f^*,g^*)\prec S^{(\delta,+)}_\delta(u,v)=
S^{(\delta,+)}_{2\delta}(\mu_0, \nu_0)\quad
 \text{ modulo } 2m.
 	\label{6.15}
	\end{equation}
We apply \eqref{6.14} with $(\hat u,\hat v)=\big(\mu(\cdot,\delta), \nu(\cdot,\delta)\big)$
	\begin{equation}
T_{2\delta}(\upmu_0, \upnu_0)=\big(\mu(\cdot,2\delta), \nu(\cdot,2\delta)\big) \prec S^{(\delta,+)}_\delta (u',v')\quad
 \text{ modulo } m
 	\label{6.17}
	\end{equation}
that together with \eqref{6.15} proves the upper bound in \eqref{eq5.20}  for $k=2$ and $c'=3\bar c$.
By using  \ref{2} of Proposition \ref{prop5.3} and Lemma
\ref{lem5.1} we similarly get the the lower bound in \eqref{eq5.20}  for $k=2$ and $c'=3\bar c$.
Theorem \ref{thm5.1} follows from the iteration of the above procedure. \qed

\appendix

\section{Proof of Proposition \ref{prop5.3}}
\label{app}

\begin{proof} Let  $H$ and $Z$ be the points
so that
\begin{equation}
	\nn
	 \int_{-\infty}^{H}u(r)dr=m,\qquad 	 \int^{+\infty}_{Z}v(r)dr=m.
	\end{equation}
	Since  $(u,v)\in \mathcal U$ for  $m_0$ small enough we have that $H<Z$.
We define
\begin{eqnarray}
\label{5.63}
f^\star=u+v\mathbf 1_{[Z,+\infty)}-u\mathbf 1_{(-\infty,H]},\qquad
 g^\star=v+u\mathbf 1_{(-\infty,H]}-v\mathbf 1_{[Z,+\infty)}.
	\end{eqnarray}
Obviously $(f^\star,g^\star)\in \mathcal B(\phi, M_0)$ and $ (u,v)\preccurlyeq (f^\star,g^\star)$.

If $r\le H$ then $F(r; f^\star) =
 \int_{\mathbb R}  u(r)dr$ $= M\ge F(r; u')$. For $r\in[H,Z]$ by using that
$(u',v') \prec (u,v)$ modulo $m$ we get
	\begin{equation}
	\nn
F(r;f^\star) =
 \int_{r}^{+\infty}  u(r)dr+m\ge F(r; u')-m+m.\qquad \forall r\in[H,Z]	\end{equation}
Finally since $g^\star=0$ for all  $r>Z$ and $(f^\star,g^\star)\in \mathcal B(\phi, M_0)$ we have that  $f^\star(r)=\phi(r)=u'(r)+v'(r)$ for all $r>Z$ and therefore
$F(r; f^\star) \ge  F(r; u')$  for all $r\ge Z$.
	Thus $(u',v') \preccurlyeq (f^\star,g^\star) $.
	
	To prove \eqref{5.56}, recalling that $S^{(\delta,+)}_\delta=G_\delta K_\delta$,  we first compare $	(\bar f^\star,\bar g^\star) :=K_\delta(f^\star, g^\star)$ with
$(\bar u,\bar v):=K_\delta(u, v)$. Let $ D^\star$, $ R^\star$ and  $ D$, $  R$ be the points such that
	\begin{equation}
\int_{ R^\star}^{+\infty}f^*(r)dr=\kappa\delta=\int^{ D^\star}_{-\infty}g^*(r)dr, \qquad
 \int_{ R}^{+\infty}u(r)dr=\kappa\delta
= \int^{ D}_{-\infty}v(r)dr.
\label{6.12}
	\end{equation}
	By definition of $K_\delta$,
\begin{eqnarray*}
\nn
&&	\bar u=u\mathbf 1_{(-\infty,R]}+v\mathbf 1_{(-\infty, D]},\qquad 	\hskip.7cm\bar v=u\mathbf 1_{[ R,+\infty)}+v\mathbf 1_{[D,+\infty)},
\\&&
\bar f^\star=f^\star\mathbf 1_{(-\infty,R^\star]}+g^\star\mathbf 1_{(-\infty,D^\star]}\qquad  \bar g^\star=f^\star\mathbf 1_{[R^\star,+\infty)}+g^\star\mathbf 1_{[D^\star,+\infty)}.
\nn
	\end{eqnarray*}
Since $ (u,v)\preccurlyeq (f^\star,g^\star)$ we have
$D^\star\le D\le  R\le  R^\star$ and
since   $(u,v)\in \mathcal U$ then for $m_0$ small enough we have that $H< D$ and $Z>R$ that implies
$\dis{\int_{R}^{ R^\star}u(r)dr\le m}$	analogously $\dis{\int^D_{ D^\star}v(r)dr \le m}$.  Let $r\le D^*$, then since $(f^\star,g^\star)\in \mathcal B(\phi, M_0)$,
	\begin{eqnarray*}
\nn
	&&\hskip-1.6cm
F(r; \bar f^*)=\int^{ D^*}_{r}  \phi (r')dr'+ \int_{ D^*}^{R^*} f^*(r')dr'.
	\end{eqnarray*}
Since  $\bar u(r')=\phi (r')$ for $r' \le D$, using the definition \eqref{5.63} we have that
\begin{eqnarray*} 	
&&\hskip-2.3cm
\int_{ D^*}^{R^*} f^* =\int^{D}_{D^*}u+\int_{ D}^{ R^*}u
+\int_{Z\wedge R^*}^{ R^*}v
\\&&
\hskip-1cm
\le\int^{D}_{D^*} \phi-
\int^{D}_{D^*} v
+\int_{ D}^{ R}  u+\int_{R}^{ R^*}u+\int_{Z\wedge R^*}^{ R^*}v
\le \int_{ D^\star}^{ R}  \bar u +2m.
\end{eqnarray*}
Thus $F(r; \bar f^*)\le F^+(r; \bar u)+2m$ for all $r\le D^*$. For $r> D^\star$
		\begin{eqnarray*}
	F(r;  \bar f^*)=
	\int^{R^\star}_{r} f^\star=\int^{R^\star}_{r}  u+
	\int^{R^\star}_{Z\wedge R^\star}  v
	\le \int^{\mathcal R}_{r}  u+2m=F(r;  \bar u)+2m.
	\end{eqnarray*}
Thus
	\begin{equation}
	\label{6.11}
(\bar f^\star,\bar g^\star) =K_\delta(f^\star, g^\star)\prec   K_\delta(u, v)=(\bar u,\bar v),\quad\text{ modulo } 2m.
		\end{equation}
We are left with the proof of the analogous inequality for the convolution with $G_\delta$. We call $C_\pm$ the point such that
	\begin{equation*}
\int_{C_+}^{+\infty} \bar f^\star(r)dr=2m,\qquad \int^{C_-}_{-\infty} \bar g^\star(r)dr=2m
	\end{equation*}
	and we let $f=\bar f^\star\mathbf 1_{(-\infty,C_+)}+\bar g^\star\mathbf 1_{(-\infty,C_-)}$ and $g=\bar f^\star\mathbf 1_{[C_+,+\infty)}+\bar g^\star\mathbf 1_{[C_-,+\infty)}$. Then, by definition $(f,g)\preccurlyeq (\bar f^\star,\bar g^\star)$ and it is not difficult to check that $(f,g)\preccurlyeq (\bar u,\bar v)$. Since $G_\delta$ is non decreasing with respect to $\preccurlyeq$  (see the proof of  Lemma \ref{lem5.1}) we have that $(G_\delta\star f,G_\delta\star g)\preccurlyeq (G_\delta\star\bar u,G_\delta\star \bar v)$. On the other hand
\begin{equation*}
F(r,G_\delta\star \bar f^*)=F(r,G_\delta\star f)+F(r,G_\delta\star (\bar f^*-f))\le
F(r,G_\delta\star \bar u)+2m.
	\end{equation*}
Thus
$G_\delta K_\delta(f^\star, g^\star)\prec  G_\delta  K_\delta(u, v)$ modulo $2m$
which proves \eqref{5.56} and thus concludes the proof of 1.

We define
\begin{equation}
\label{5.67}
f_\star=u'+v'\mathbf 1_{(-\infty,Z')}-u'\mathbf 1_{[H',+\infty)},\qquad
g_\star=v'+u'1_{[H',+\infty)}-v'\mathbf 1_{(-\infty,Z')},
	\end{equation}
where  $H'$ is such that $\dis{ \int^{+\infty}_{H'}u'(r)dr=m}$ and  $Z'$  is such that $\dis{  \int_{-\infty}^{Z'}v'(r)dr=m}$.
By definition  $(f_\star,g_\star)\in \mathcal B(\phi, M_0)$ and $(f_\star,g_\star)\preccurlyeq  (u',v')$. We next observe that for $r\le Z'$
	\begin{equation}
	\nn
F(r; f_\star)=\int_{-\infty}^{+\infty} f_\star-\int^r_{-\infty} [u'+v']=M_0
-\int^r_{-\infty} [u+v]\le M_0
-\int^r_{-\infty} u= F(r; u).
	\end{equation}
For $r\in[Z',H']	$ we have
	\begin{equation}
	\nn
F(r;f_\star) =
 \int_{r}^{+\infty}  u'(r)dr-m\le F(r; u).
  \end{equation}
And finally for $ r\ge H'$, $F(r;f_\star)=0 \le  F(r; u)$, which concludes the proof that $ (f_\star,g_\star) \preccurlyeq (u,v) $.
To prove \eqref{6.10} we first write
	\begin{eqnarray}
		\nn
&&G_\delta\star	f_\star=G_\delta\star	u'+G_\delta\star	(v'\mathbf 1_{(-\infty,Z')})-G_\delta\star	(u'\mathbf 1_{[H',+\infty)}),
\\&& \label{6.13}
G_\delta\star	g_\star=G_\delta\star	v'+G_\delta\star	(u'1_{[H',+\infty)})-G_\delta\star	(v'\mathbf 1_{(-\infty,Z']}).
	\end{eqnarray}
Let   $  D'$, $  R'$,
 $  D_\star$ and $  R_\star$ be the points such that
	\begin{eqnarray*}
&&\int_{R_\star}^{+\infty}G_\delta\star	f_\star =\kappa\delta=\int^{D_\star}_{-\infty}G_\delta\star g_\star , \qquad
 \int_{ R'}^{+\infty}G_\delta\star	u' =\kappa\delta
= \int^{ D'}_{-\infty}G_\delta\star v'.
	\end{eqnarray*}
From the fact that the convolution with $G_\delta$ preserves the inequality
we have
$ D'\le  D_\star\le R_\star\le  R'$.
Furthermore using \eqref{6.13} we get
	\begin{equation}
		\label{5.70}
\int_{ R_\star}^ { R'}
G_\delta\star	u'(r)dr\le m,\qquad \int^{D_\star}_{ D'}
G_\delta\star	v'(r)dr\le m.
	\end{equation}
Recalling  the definition of $K_\delta$ we call  $	(\bar f_\star, \bar g_\star) :=K_\delta(G_\delta\star f_\star, G_\delta\star g_\star)$ and
$(\bar u',\bar v'):=K_\delta(G_\delta\star u', G_\delta\star v')$.
For $r\le D'$, using that $(f_\star,g_\star)$ and $(u',v')$ are both in the set $\mathcal B(\phi, M_0)$ and \eqref{5.70}  we have for $r\le  D'$
	\begin{eqnarray}
		\nn
 &&
 \hskip-.6cm
 \int_r^{+\infty}\bar u'=\int_r^{D'} G_\delta\star \phi+\int_{D'}^{D_\star}
 G_\delta\star u'+\int_{D_\star}^{R'}G_\delta\star u'
 \\&&\nn
 \le \int_r^{D_\star} G_\delta\star \phi +\int_{D_\star}^{R_\star}G_\delta\star u'+m
  \\&&\nn
  \le  \int_r^{D_\star} G_\delta\star \phi +\int_{D_\star}^{R_\star}G_\delta\star f_\star+2m =F(r; \bar f_\star)+2m.
	\end{eqnarray}
Analogously, for $r> D'$
\begin{eqnarray}
		\nn
 &&
 \hskip-.6cm
 \int_r^{+\infty}\bar u' \le \int_r^{D_\star}  G_\delta\star u'+\int_{D_\star}^{R_\star}G_\delta\star u' +m
 \\&&\nn
 \le \int_r^{D_\star} G_\delta\star \phi +\int_{D_\star}^{R_\star}G_\delta\star f_\star+2m=F(r; \bar f_\star)+2m.
	\end{eqnarray}
This proves \eqref{6.10} and concludes the proof of the Proposition.
\end{proof}

\section*{Acknowledgements}
We thank very useful discussions with E. Presutti.
We thank warm hospitality at the Gran Sasso Science Institute in L'Aquila.

\end{document}